\documentclass{nlaauth}

\usepackage{amsmath}
\usepackage{amssymb}
\usepackage{amsfonts}
\usepackage{lmodern}
\usepackage{latexsym}
\usepackage{mathrsfs}
\usepackage{bm}
\usepackage{verbatim}
\usepackage{setspace}
\usepackage{cite}
\usepackage{multirow}
\usepackage{stmaryrd}
\usepackage{subcaption}
\usepackage{adjustbox}
\usepackage{algorithm}
\usepackage{algorithmic}

\usepackage{tikz}
\usetikzlibrary{arrows,mindmap,shadows}

\usepackage[colorlinks,citecolor=red,linkcolor=blue,urlcolor=black,pdftex]{hyperref}
\usepackage[capitalize]{cleveref}

\usepackage{tikz} 
\usetikzlibrary{decorations.markings}
\newcommand{\vertiii}[1]{{\left\vert\kern-0.25ex\left\vert\kern-0.25ex\left\vert #1 
    \right\vert\kern-0.25ex\right\vert\kern-0.25ex\right\vert}}
\theoremstyle{plain}
\newtheorem{theorem}{Theorem}
\newtheorem{lemma}{Lemma}
\newtheorem{remark}{Remark}
\newtheorem{corollary}{Corollary}

\begin{document}

\title{Error Estimates for the Arnoldi Approximation of a Matrix Square Root}

\author{James H. Adler\affil{1}, Xiaozhe Hu\affil{1}, and 
  Wenxiao Pan\affil{2}, and Zhongqin Xue\affil{1}\corrauth}

\address{\affilnum{1}Department of Mathematics, Tufts University, Medford, MA 02155, USA\break\affilnum{2}Department of Mechanical Engineering, University of Wisconsin-Madison, Madison, WI 53706, USA }

\corraddr{Department of Mathematics, Tufts University, Medford, MA 02155, USA}

\begin{abstract}
The Arnoldi process provides an efficient framework for approximating functions of a matrix applied to a vector, i.e., of the form $f(M)\bm{b}$, by repeated matrix-vector multiplications. In this paper, we derive error estimates for approximating the action of a matrix square root using the Arnoldi process, where the integral representation of the error is reformulated in terms of the error for solving the linear system  $M\bm{x}=\bm{b}$. The results extend the error analysis of the Lanczos method for Hermitian matrices in [Chen et al., SIAM J. Matrix Anal. Appl., 2022] to non-Hermitian cases and provide an improved bound for the Hermitian case.  Furthermore, in practical settings, the matrix may only be available via approximate or structured representations. Motivated by this, we extend the analysis and establish a generalized error bound for perturbed matrices.
The numerical results on matrices with different structures demonstrate that our theoretical analysis yields a reliable upper bound. 
Finally, simulations on large-scale matrices arising in particulate suspensions, represented in hierarchical matrix form, validate the effectiveness and practicality of the approach.
\end{abstract}


\maketitle

\section{Introduction}
Given a matrix $M\in \mathbb{C}^{n\times n}$, its square roots are the solutions of 
$X^2=M.$
If $M$ is positive definite, then $M$ admits a unique positive definite square root \cite{higham2008functions}, denoted as $M^{1/2}$. 
Matrix square roots frequently arise in 
stochastic simulation \cite{ColloidsReview_COCIS2019}, state estimation \cite{asl2019adaptive}, image reconstruction \cite{bardsley2012mcmc}, 
and Bayesian optimization \cite{pleiss2020fast,hernandez2014predictive}. 
Hence, the computation of a matrix square root is a fundamental problem in numerical linear algebra. 

There are a list of approaches in the literature for computing the square root of a matrix. Cholesky decomposition and eigenvalue decomposition provide useful tools for symmetric positive definite (SPD) matrices \cite{watkins2004fundamentals}. For general matrices, the Schur decomposition transforms the matrix into a triangular form, where the square root can be computed explicitly \cite{higham2008functions}. 
Additionally, Newton-like iterations \cite{higham2008functions} provide stable and efficient implementations for approximating the matrix square root.  We also note that there are several methods that have recently been proposed for computing $M^{1/2}$ (or $M^{-1/2}$) explicitly, such as Steffensen-type iterations, doubling algorithms, and mixed-precision techniques \cite{unal2026steffensen,gao2026mixed,kudo2025fast,guo2025computing,bini2026component}.
However, these methods typically suffer from $\mathcal{O}(n^3)$ complexity and $\mathcal{O}(n^2)$ storage for large-scale matrices.
Moreover, when $M$ is sparse or structured, its square root is typically dense, making direct
computation and storage costly or even infeasible for large-scale problems. 
Consequently, rather than computing the square root of a matrix explicitly, 
one can instead compute
the action of a matrix square root on a vector. This is also the case for general matrix functions. Krylov subspace methods such as the Arnoldi iteration, therefore provide a powerful framework, as they only rely on repeated matrix-vector multiplications to construct the Krylov subspace. For general dense matrices, the overall computational cost of the Arnoldi process is about $\mathcal{O}(kn^2)$, where $k$ is the number of Arnoldi iterations. 
When the matrix is Hermitian, the Arnoldi process simplifies to the Lanczos algorithm \cite{chen2024lanczos}, characterized by a three-term recurrence.

There has been growing interest  in the study and discussion of Krylov methods for approximating the action of matrix functions. An early study by Saad \cite{saad1992analysis} relates the error of Krylov subspace approximations for $f(M)\mathbf{b}$ to polynomial approximation.
If $M$ is Hermitian, then one obtains the near-optimality property of Lanczos, that is, the error after $k$ iterations is bounded by that of the best polynomial of degree less than $k$ uniformly approximating $f(x)$ over the spectrum of $M$; see, e.g. \cite{musco2018stability,chen2022error}.
In \cite{diele2009error}, the error of polynomial Krylov approximations for matrix functions is represented in integral forms, and \textit{a posteriori} error estimates for exponential-like and trigonometric matrix functions are provided. 
In \cite{eiermann2006restarted,ilic2010restarted}, the authors express the errors of Krylov subspace approximations in terms of divided differences, providing insights into the design of restart schemes.
For general analytic functions $f$ admitting an integral representation, \cite{frommer2014efficient} provides an integral representation for the error of the Arnoldi approximation to $f(M)\bm{b}$, and proposes a restarting strategy based on evaluating this error function via numerical quadrature. 
Based on \cite{frommer2014efficient,frommer2014convergence}, Frommer and Schweitzer \cite{frommer2016error} derive error bounds and estimates of the Krylov subspace
approximations $f(M)\bm{b}$ for Stieltjes functions.
Using the Cauchy integral formula, Chen et al. \cite{chen2022error} establish \textit{a priori} and \textit{a posteriori} error bounds for the Lanczos method by reformulating the approximation error in terms of the error of solving a shifted linear system. In addition, Amsel et al. \cite{amsel2024nearly} explore the near optimality of the Lanczos methods for rational functions with real poles lying outside the spectral interval of $M$. Recently, \cite{schweitzer2025near} proves that, for Stieltjes and related functions, the Lanczos method exhibits, what they call, near-instance optimality. There are also several works that study Krylov approximations in the context of rational approximations \cite{frommer2009error,frommer2008stopping,frommer20132} and matrix exponentials \cite{beckermann2009error,stewart1996error,hochbruck1997krylov,druskin1998using}. Most of these results are established under the assumption that the matrices are Hermitian, while in many practical applications, the matrices of interest are non-Hermitian, which motivates us to extend the results to the non-Hermitian setting.



In this work, we present error bounds of the Arnoldi process specifically for evaluating a matrix square root, extending the results in \cite{chen2022error} to the case of positive-definite non-Hermitian matrices. In this context, we say that $M$ is positive-definite if $\operatorname{Re}(\bm{x}^*M\bm{x})>0$ for any unit vector $\bm{x}\in\mathbb{C}^n$. The main idea is to utilize the integral representation of the error in the iterations of the Arnoldi process, and then relate it to the error for solving the linear system $M\bm{x}=\bm{b}$. 
Compared to Hermitian matrices, non-Hermitian matrices exhibit complex spectral behavior: they are not necessarily diagonalizable, and even when they are, the eigenvectors are typically not orthogonal and the eigenvalues may be complex. 
To address these difficulties, 
we relate the modulus of a function in the complex plane to a real-valued function of the modulus, reducing the analysis to the positive real line.   This also allows us to derive a sharper bound for the Hermitian case by using the average of eigenvalues, which improves on the results in {\cite{chen2022error,amsel2024nearly}.   Additionally, for large-scale problems, we derive an error bound for approximating the matrix square root based on a data-sparse representation of the matrix.  
In numerical experiments, we investigate  matrices with different structures to verify the theoretical results.
Finally, we validate the effectiveness of our approach on large-scale mobility matrices arising from particulate suspension simulations. 
To improve computational efficiency and reduce storage cost, the matrices are preprocessed through adaptive cross approximation and stored in the form of hierarchical matrices.

The remainder of this paper is organized as follows. 
In Section \ref{Arnoldi}, 
we introduce the Arnoldi process for matrix functions. The error estimation of 
the Arnoldi process for computing a square root is provided in Section \ref{main}.  In Section \ref{numerical},  numerical examples are presented to
validate the theoretical results discussed in this paper. Finally, concluding remarks are provided in Section \ref{sum}.

\section{Matrix functions and Arnoldi iterations}\label{Arnoldi} 
There are different ways to define a function of a matrix, depending on the properties of the matrix and the type of function. In this work, we focus on two of them: the definition via the Jordan canonical form and the Cauchy integral representation.  

Given a square matrix $M \in \mathbb{C}^{n \times n}$, assume its Jordan canonical form is given by 
$$
M = Z J Z^{-1},
$$
where $Z$ is a nonsingular matrix, and  $J=\operatorname{diag}\left(J_1, J_2, \ldots,\right.\left.J_p\right)$ with   $J_k\in\mathbb{R}^{m_k\times m_k}$, $\sum_{i=1}^p m_k = n$, being a Jordan block associated with the eigenvalue $\lambda_k$, $1\leq k \leq p$. Then, the matrix function $f(M)$ is defined as
$$
f(M):=Z f(J) Z^{-1}=Z \operatorname{diag}\left(f\left(J_1\right),f\left(J_2\right),\cdots ,f\left(J_p\right)\right) Z^{-1},
$$  
where $f(J_k)$ takes the form:
\begin{align}\label{de_matrixfun}
f\left(J_k\right):=\left[\begin{array}{cccc}
f\left(\lambda_k\right) & f^{\prime}\left(\lambda_k\right) & \ldots & \frac{f^{\left(m_k-1\right)}\left(\lambda_k\right)}{\left(m_k-1\right)!} \\
& f\left(\lambda_k\right) & \ddots & \vdots \\
& & \ddots & f^{\prime}\left(\lambda_k\right) \\
& & & f\left(\lambda_k\right)
\end{array}\right].
\end{align}

Alternatively, when $f$ is analytic on a domain containing the spectrum of $M$, one can also define $f(M)$ using the Cauchy integral formula:
$$
f(M) = -\frac{1}{2\pi i} \int_{\mathcal{C}} f(z) (M- zI)^{-1} \, \mathrm{d}z,
$$
where $\mathcal{C}$ is a closed contour enclosing the spectrum of $M$. The two definitions are equivalent when $f$ is analytic \cite{higham2008functions}.

In practice, we are primarily interested in the action of a matrix function $f(M)$ on a vector $\bm{b}$. To achieve this efficiently, we employ the Arnoldi process, which is described in \cref{Arn_algor}. 
Given $\bm{b} \in \mathbb{C}^n$, this yields
\begin{align}\label{Arnoldipro}
    M Q_k = Q_k H_k + h_{k+1,k}\bm{q}_{k+1}\bm{e}_k^T,
\end{align}
where $H_k=[h_{i,j}] \in \mathbb{C}^{k \times k}$ is an upper Hessenberg matrix satisfying $Q_k^* M Q_k = H_k$, $\bm{e}_k$ is the $k$-th  canonical basis vector, and $Q_k$ is a matrix with orthonormal columns $\bm{q}_{i}$, which form a  basis for the $k$-th Krylov space,
$$\mathcal{K}_k(M, \bm{b}) = \text{span}\{\bm{b}, M\bm{b}, M^2\bm{b}, \dots, M^{k-1}\bm{b}\}.$$
Then, the $k$-th Arnoldi iteration for approximating $f(M)\bm{b}$ is given by
\begin{align*}
f(M)\bm{b}\approx\operatorname{Arn}_k(f ; M, \bm{b}):=Q_k f(H_k) Q_k^*\bm{b}=\|\bm{b}\|Q_k f(H_k) \bm{e}_1.
\end{align*}

\begin{algorithm}[h!]
\caption{Arnoldi Process}
\begin{algorithmic}[1]\label{Arn_algor}
\STATE \textbf{Input:} Matrix $M \in \mathbb{C}^{n \times n}$, vector $\bm{b}\in \mathbb{C}^{n}$, and iteration number $k$
\STATE \textbf{Initialize:} Set $\bm{q}_1 = \bm{b}/\|\bm{b}\|$ and $Q_k=[\bm{q}_1]$
\FOR{$j = 1$ to $k$}
\STATE $\bm{w} = M \bm{q}_j$
\FOR{$i = 1, \dots, j$} 
\STATE $h_{i,j} = \bm{q}_i^* \bm{w}$
\STATE $\bm{w} = \bm{w} - h_{i,j} \bm{q}_i$
\ENDFOR
\STATE $h_{j+1,j} = \|\bm{w}\|$
\IF{$h_{j+1,j} = 0$}
\STATE Stop
\ENDIF
\STATE $\bm{q}_{j+1} = \bm{w} / h_{j+1,j}$
\STATE Append $\bm{q}_{j+1}$ to $Q_k = [\bm{q}_1, \dots, \bm{q}_{j+1}]$
\ENDFOR
\end{algorithmic}
\end{algorithm}

Note that when $f(x) = 1/x$, the Arnoldi approximation reduces to the Full Orthogonalization Method (FOM) \cite{saad2003iterative} for solving the linear system $M\bm{x} = \bm{b}$, with initial guess $\bm{x}_0 = \bm{0}$. For the rest of this paper, whenever FOM is referred, we always assume the initial guess is zero.

\section{Error bound}\label{main} 
In this section, we study error bounds of the Arnoldi process for approximating general matrix functions and then focus on the analysis of the matrix square root.
\subsection{Integral representation of the error}
Let $f$ be analytic along and inside a contour $\mathcal{C}$ that encloses the spectrum of both $M$ and $H_k$. By the Cauchy integral representation of $f(H_k)$, the $k$-th Arnoldi approximation to $f(M)\bm{b}$ is given by:
\begin{align*}
\operatorname{Arn}_k(f ; M, \bm{b})=\|\bm{b}\|Q_k f(H_k) \bm{e}_1=- \frac{\|\bm{b}\|}{2\pi i} \int_\mathcal{C} f(z) Q_k (H_k-z I)^{-1} \bm{e}_1\, \mathrm{d}z.
\end{align*}
The integral involves solving a  shifted linear system $(M - zI)\bm{x}=\bm{b}$ using FOM, yielding the approximation $\|\bm{b}\| Q_k (H_k - zI)^{-1} \bm{e}_1$. 
The corresponding error and residual are defined as:
$$
\bm{\xi}_z^k := (M - zI)^{-1}\bm{b} - Q_k (H_k - zI)^{-1} \|\bm{b}\|\bm{e}_1\quad\text{and}\quad \bm{r}_z^k :=(M - zI)\bm{\xi}_z^k.
$$
Then, we represent the error of the Arnoldi process for 
$f(M)\bm{b}$ by:
\begin{align}\label{err}
   f(M)\bm{b}-\operatorname{Arn}_k(f ; M, \bm{b}) = f(M)\bm{b} - \|\bm{b}\|Q_k f(H_k) \bm{e}_1 = -\frac{1}{2\pi i} \int_\mathcal{C} f(z) \bm{\xi}_z^k \, \mathrm{d}z.
\end{align}

The following lemma establishes an explicit formula for the FOM residual associated with $(M - zI)\bm{x} = \bm{b}$. This result follows directly from the collinearity property of FOM residuals (see, e.g., \cite{saad2003iterative}).
\begin{lemma}\label{est:res}
For $z\in \mathbb{C}$ lying on a closed contour $\mathcal{C}$ that encloses the spectrum of $H_k$, 
$$
\bm{r}_z^k = \left( \frac{(-1)^k}{\det(H_k - zI)} \prod_{j=1}^k h_{j+1,j} \right) \|\bm{b}\| \, \bm{q}_{k+1}.
$$
\end{lemma}

\begin{proof}
As shown in \cite[Proposition 6.7]{saad2003iterative}, the FOM residual $\bm{r}_z^k$ is collinear to the next Krylov basis vector $\bm{q}_{k+1}$ and satisfies
$$\bm{r}_z^k = -h_{k+1,k} [\bm{e}_k^T (H_k - zI)^{-1} \bm{e}_1] \Vert{}\bm{b}\Vert{} \bm{q}_{k+1}.$$
Note that $H_k - zI$ is an upper Hessenberg matrix. Using Cramer's rule, we have
$$
\bm{e}_k^T (H_k - zI)^{-1} \bm{e}_1 = \frac{\text{cof}(H_k - zI)_{1k}}{\det(H_k - zI)}= \frac{(-1)^{k+1}}{\det(H_k - zI)} \prod_{j=1}^{k-1} h_{j+1,j},
$$
where $\text{cof}(H_k - zI)_{1k}$ denotes the $(1,k)$-cofactor of $H_k - zI$. Substituting this into the expression for $\bm{r}_z^k$ gives the desired result.
\end{proof}
Applying \cref{est:res}, we establish the following result, which connects the errors and residuals of the Arnoldi approximations for different shifts.
\begin{lemma}\label{rela_err}
Assuming $M$ and $M - zI$ are invertible for some $z \in \mathbb{C}$, we obtain
$$
\bm{\xi}_z^k = \frac{\det(H_k)}{\det(H_k-zI)}(M-zI)^{-1}M \bm{\xi}_0^k\quad \text{and}\quad
\bm{r}_z^k = \frac{\det(H_k)}{\det(H_k-zI)}\bm{r}_0^k.
$$
\end{lemma}

\begin{proof}
The proof follows a similar argument to that of Corollary 2.4 in \cite{chen2022error}.
\end{proof}

Then, we have the following Theorem.
\begin{theorem}\label{error_ref}
Given a nonsingular matrix $M \in \mathbb{C}^{n \times n}$ and a vector $\bm{b} \in \mathbb{C}^n$, let the function $f$ be analytic on and inside a closed contour $\mathcal{C}$ enclosing the spectrum of $M$ and $H_k$, where $H_k$ is the upper Hessenberg matrix generated from the Arnoldi process for $M$ and $\bm{b}$. Then, the error of the Arnoldi iteration for $f (M)\bm{b}$ can be expressed as
\begin{align}\label{th:err}
    f(M) \bm{b}-\operatorname{Arn}_k(f ; M, \bm{b})=\left(-\frac{1}{2 \pi i} \int_{\mathcal{C}} f(z) \frac{\det(H_k)}{\det(H_k-zI)}(M-zI)^{-1}M \, \mathrm{d}z\right) \bm{\xi}_0^k.
\end{align}
\end{theorem}
\begin{proof}
    The result is immediate from \cref{err} and \cref{rela_err}:
    \begin{align*}
        f(M) \bm{b}-\operatorname{Arn}_k(f ; M, \bm{b}) & = -\frac{1}{2\pi i} \int_\mathcal{C} f(z) \bm{\xi}_z^k \, \mathrm{d}z\\
        & = -\frac{1}{2\pi i} \int_\mathcal{C} f(z) \frac{\det(H_k)}{\det(H_k-zI)}(M-zI)^{-1}M \bm{\xi}_0^k\, \mathrm{d}z.
    \end{align*}
\end{proof}

\begin{remark}
   Here we assume that the function $f$ is analytic on and inside a closed contour $\mathcal{C}$ enclosing the spectrum of $M$ and $H_k$. However, the theorem holds under the sufficient condition that $\mathcal{C}$ encloses the field of values of $M$. 
   
   Similarly to the results in \cite{frommer2009error,frommer2014efficient},  \cref{error_ref} shows that for any nonsingular $M$, the error of the Arnoldi iteration for computing $f(M)\bm{b}$ is reformulated as the error in solving the linear system $M\bm{x} = \bm{b}$. This further simplifies the contour integral. 
   For a more general setting, the error admits a bound related to the error of solving shifted linear systems of the form $(M - wI)\bm{x} = \bm{b}$, where $w$ lies outside the spectrum of both $H_k$ and $M$.

\end{remark}
\subsection{Matrix square root}\label{32}
We now turn to deriving error estimates of the Arnoldi process for approximating a matrix square root, that is $f(x) = x^{1/2}$. To ensure the existence of a unique positive definite square root, we assume that $M$ is positive-definite. 



\begin{theorem}\label{error_square}
Given a positive-definite matrix $M \in \mathbb{C}^{n \times n}$ and a vector $\bm{b} \in \mathbb{C}^n$, the error of the Arnoldi process for approximating the matrix square root is bounded by
\begin{align*}
\|M^{1 / 2} \bm{b}-\operatorname{Arn}_k(x^{1 / 2} ; M, \bm{b})\|\leq\frac{1}{\pi} \left(\int_0^{\infty}x^{1 / 2} \prod_{j=1}^k\left|\frac{\lambda_j(H_k)}{\lambda_j(H_k)+x}\right| \, \mathrm{d}x\right)\left\|\bm{\xi}_0^k\right\|,
\end{align*}
where  $\lambda_1(H_k)\geq\lambda_2(H_k)\geq\cdots\geq\lambda_k(H_k)$ are the eigenvalues of $H_k$.
\end{theorem}
\begin{proof}
Define $\mathcal{C}$ as a keyhole contour, consisting of a large circular arc $\mathcal{C}_R$ of radius $R$ centered at the origin, a small  circular arc $\mathcal{C}_\epsilon$ of radius $\epsilon$ around the origin, and two horizontal line segments $\mathcal{C}^+_{\delta}$ and $\mathcal{C}^-_{\delta}$, which run at a small vertical distance $0<\delta<\epsilon$ above and below the branch cut, respectively (see \cref{fig:keyhole_contour} for details). 

\begin{figure}[ht] 
    \centering 
\begin{tikzpicture}[>=stealth,scale=0.8]
\def\R{3} 
\def\e{1} 
\def\gap{12} 

\draw[thick,blue,dashed,postaction={decorate},decoration={markings,mark=at position 0.25 with {\arrow[thick,blue]{<}}}] 
  ([shift={(-6,0)}] 0.5*\gap:\R) arc[start angle=180-0.5*\gap,end angle=-180+0.5*\gap,radius=\R] node[below left,xshift=50mm,yshift=-10mm] {$\mathcal{C}_R$};

\draw[thick,blue,dashed,postaction={decorate},decoration={markings,mark=at position 0.75 with {\arrow[thick,blue]{>}}}] 
  ([shift={(-2,0)}] \gap:\e) arc[start angle=180-\gap,end angle=-180+\gap,radius=\e] node[below left,xshift=18mm,yshift=-4mm] {$\mathcal{C}_\epsilon$};

\draw[thick,-,black] (0,0) -- (50:\R) 
    node[midway,above right,xshift=1mm,yshift=-1mm] {$R$};

\draw[thick,-,black] (0,0) -- (20:\e) 
    node[midway,above right,xshift=-1mm,yshift=-0.5mm] {$\epsilon$};

\draw[thick,purple,<-] (-\e,0.25) -- (-\R,0.25) node[midway,above,xshift=-1mm] {$\mathcal{C}_\delta^+$};

\draw[thick,purple,->] (-\e,-0.25) -- (-\R,-0.25) node[midway,below,xshift=-1mm] {$\mathcal{C}_\delta^-$};

\draw[<->,thick] (-2.2,0.25) -- (-2.2,-0.25);
\node[right] at (-2.2,0) {2$\delta$};

\node[black,below right] at (-0,0) {O};

\draw[->] (-\R-1,0) -- (\R+0.5,0) node[right] {$x$};
\draw[->] (0,-\R-0.5) -- (0,\R+0.5) node[above] {$y$};


\end{tikzpicture}
\caption{Keyhole contour with outer circle $\mathcal{C}_R$, inner circle $\mathcal{C}_\epsilon$, and offset $\delta$ above/below the branch cut.} 
\label{fig:keyhole_contour} 
\end{figure}

Using properties of the determinant, we have
$$
\left|\frac{\det(H_k)}{\det(H_k-zI)}\right|
=\left|\prod_{j=1}^k \frac{\lambda_j\left(H_k\right)}{\lambda_j(H_k-zI)}\right|=\prod_{j=1}^k \left|\frac{\lambda_j\left(H_k\right)}{\lambda_j(H_k)-z}\right|.
$$
Taking the 2-norm on both sides of \eqref{th:err} with $f(z)=z^{1/2}$, we obtain
\begin{align*}
&\left\|M^{1 / 2} \bm{b}-\operatorname{Arn}_k(f ; M, \bm{b})\right\| \\
\leq
&\left(\frac{1}{2 \pi} \int_{\mathcal{C}_R}|z^{1 / 2}| \prod_{j=1}^k \left|\frac{\lambda_j\left(H_k\right)}{\lambda_j(H_k)-z}\right|\left\|(M-zI)^{-1}M\right\| \, |\mathrm{d}z|\right)\left\|\bm{\xi}_0^k\right\|\\
&-\left(\frac{1}{2 \pi} \int_{\mathcal{C}_\epsilon}|z^{1 / 2}| \prod_{j=1}^k \left|\frac{\lambda_j\left(H_k\right)}{\lambda_j(H_k)-z}\right|\left\|(M-zI)^{-1}M\right\| \, |\mathrm{d}z|\right)\left\|\bm{\xi}_0^k\right\|\\
&+\left(\frac{1}{2 \pi} \int_{\mathcal{C}_\delta^+}|z^{1 / 2}|\prod_{j=1}^k \left|\frac{\lambda_j\left(H_k\right)}{\lambda_j(H_k)-z}\right|\left\|(M-zI)^{-1}M\right\| \, |\mathrm{d}z|\right)\left\|\bm{\xi}_0^k\right\|\\
&-\left(\frac{1}{2 \pi} \int_{\mathcal{C}_\delta^-}|z^{1 / 2}|\prod_{j=1}^k \left|\frac{\lambda_j\left(H_k\right)}{\lambda_j(H_k)-z}\right|\left\|(M-zI)^{-1}M\right\| \, |\mathrm{d}z|\right)\left\|\bm{\xi}_0^k\right\|.
\end{align*}
As the outer radius $R \to \infty$, the integral over the large circular arc vanishes for $k\geq1$ since $|z^{1 / 2}|=\mathcal{O}\left(R^{1 / 2}\right)$, $\left|\frac{\lambda_j(H_k)}{\lambda_j(H_k)-z}\right|= \mathcal{O}(R^{-1})$, and $\|M(M-zI)^{-1}\| = \mathcal{O}(R^{-1})$.  Similarly, as the inner radius $\epsilon \to 0$, the integral goes to zero. 
Thus, as $\epsilon \to 0$ and $R \to \infty$, the contributions to the integral come from $\mathcal{C}^+_{\delta}$ and $\mathcal{C}^-_{\delta}$. Also, as $\epsilon \to 0$, $\delta$ approaches 0. Then, we have
\begin{align*}
&\left\|M^{1 / 2} \bm{b}-\operatorname{Arn}_k(f ; M, \bm{b})\right\| \\ 
\leq&\lim_{\substack{\epsilon \to 0 \\ R \to \infty}}\left(\frac{1}{2 \pi} \int_{-R}^{-\sqrt{\epsilon^2-\delta^2}}\left|(x + \delta i)^{1 / 2}\right| \prod_{j=1}^k\left|\frac{\lambda_j(H_k)}{\lambda_j(H_k)-(x+\delta i)}\right|  \left\|\left(M-(x+\delta iI)\right)^{-1}M\right\| \mathrm{~d} x\right)\left\|\bm{\xi}_0^k\right\|\\
&-\left(\frac{1}{2 \pi} \int^{-R}_{-\sqrt{\epsilon^2-\delta^2}}\left|(x - \delta i)^{1 / 2}\right| \prod_{j=1}^k\left|\frac{\lambda_j(H_k)}{\lambda_j(H_k)-(x-\delta i)}\right|\left\|\left(M-(x-\delta iI)\right)^{-1}M\right\| \, \mathrm{d}x\right)\left\|\bm{\xi}_0^k\right\|\\
=&\frac{1}{\pi} \left(\int_0^{\infty}x^{1 / 2} \prod_{j=1}^k\left|\frac{\lambda_j(H_k)}{\lambda_j(H_k)+x}\right| \left\|(M+xI)^{-1}M\right\| \, \mathrm{d}x\right)\left\|\bm{\xi}_0^k\right\|.
\end{align*}
Next, we prove that $\left\|(M+xI)^{-1}M\right\|<1$. This is equivalent to showing that for any non-zero vector $v\in\mathbb{C}^n$ and $x>0$,
$$\left\|(M+xI)^{-1}M\bm{v}\right\|<\left\|\bm{v}\right\|.$$
To prove this, let $\bm{u}=(M+xI)^{-1}M\bm{v}$, and note that $M(\bm{u}-\bm{v})=-x\bm{u}.$
Taking the inner product with $\bm{u}-\bm{v}$, we have
$$
-x\operatorname{Re}\left(\bm{u},\bm{u}-\bm{v}\right)=\operatorname{Re}\left(M(\bm{u}-\bm{v}),\bm{u}-\bm{v}\right)>0.
$$
Thus, we have
\begin{align*}
    0>\operatorname{Re}\left(\bm{u},\bm{u}-\bm{v}\right)=\left\|\bm{u}\right\|^2-\operatorname{Re}\left(\bm{u},\bm{v}\right)\geq\left\|\bm{u}\right\|^2-|\left(\bm{u},\bm{v}\right)|\geq\left\|\bm{u}\right\|^2-\left\|\bm{u}\right\|\left\|\bm{v}\right\|
\end{align*}
which implies $\left\|\bm{u}\right\|<\left\|\bm{v}\right\|$ and, therefore, $\left\|(M+xI)^{-1}M\right\|<1$.  Thus,
    \begin{align*}
        &\left\|M^{1 / 2} \bm{b}-\operatorname{Arn}_k(f ; M, \bm{b})\right\|
        \leq\frac{1}{\pi} \left(\int_0^{\infty}x^{1 / 2} \prod_{j=1}^k\left|\frac{\lambda_j(H_k)}{\lambda_j(H_k)+x}\right|  \, \mathrm{d}x\right)\left\|\bm{\xi}_0^k\right\|,
    \end{align*}
    which completes the proof. 
\end{proof}
\begin{remark}\label{stop}
In the Arnoldi process, the Hessenberg matrix $H_k$ is generated naturally, which means the FOM solution to $M\bm{x}=\bm{b}$ can be obtained directly. If $M$ is well-conditioned, the exact solution $M^{-1}\bm{b}$ and thereby the FOM error norm $\left\Vert{}\bm{\xi}_0^k\right\Vert{}$ can be computed inexpensively. The error bound in \cref{error_square} then provides an a posteriori error estimate that can be computed efficiently, especially for small $k$. This motivates the following practical stopping criterion:
    \begin{align*}
    \frac{1}{\pi}\left\|\bm{\xi}_0^k\right\| \int_0^{\infty}x^{1 / 2} \prod_{j=1}^k\left|\frac{\lambda_j(H_k)}{\lambda_j(H_k)+x}\right| \, \mathrm{d}x\leq\textit{tol},
    \end{align*}
    where $\textit{tol}$ is a given tolerance.
    
    Otherwise, if computing $\left\|\bm{\xi}_0^k\right\|$ is prohibitive, we can instead use the FOM residual in practice, since $\left\|\bm{\xi}_0^k\right\| \leq \left\|M^{-1}\right\| \left\|\bm{r}_0^k\right\| $.  The FOM residual can be obtained directly at negligible cost from the Arnoldi procedure, which suggests an alternative stopping criterion: 
    \begin{align*}
    \frac{1}{\pi}\left\|\bm{r}_0^k\right\| \int_0^{\infty}x^{1 / 2} \prod_{j=1}^k\left|\frac{\lambda_j(H_k)}{\lambda_j(H_k)+x}\right| \, \mathrm{d}x\leq\textit{tol}.
    \end{align*}
    Since the factor $\Vert{}M^{-1}\Vert{}$ is droped, this is only a heuristic indicator that can be used in practice and not a rigorous bound.
\end{remark}

Next, building on the a posteriori error estimate in \cref{error_square}, we derive an a priori error estimate, which is summarized in the following theorem.
\begin{theorem}\label{error_lo}
    Given a positive-definite matrix $M \in \mathbb{C}^{n \times n}$ and a vector $\bm{b} \in \mathbb{C}^n$, for $k\geq2$, the error of the Arnoldi process for approximating the matrix square root is bounded by
\begin{align*}
 \left\|M^{1 / 2} \bm{b}-\operatorname{Arn}_k(f ; M, \bm{b})\right\|\leq\frac{\Gamma\!(3/4)}{2^{1/4}\pi}\frac{2k}{2k-3}(\sigma_{\max}(M))^{\frac{3}{2}}k^{-\frac{3}{4}}\left\|\bm{\xi}_0^k\right\|,
\end{align*}
where 
$\Gamma\!(\cdot)$ is the standard Gamma function, and $\sigma_{\max}(M)$ denotes the largest singular value of $M$. 
\end{theorem}
\begin{proof}
From Theorem \ref{error_square}, we know that
\begin{align*}
\left\|M^{1 / 2} \bm{b}-\operatorname{Arn}_k(f ; M, \bm{b})\right\| \leq\frac{1}{\pi}\int_0^{\infty}x^{1 / 2} \prod_{j=1}^k\left|\frac{\lambda_j(H_k)}{\lambda_j(H_k)+x}\right|\, \mathrm{d}x.
\end{align*}
Consider any $w\in \mathbb{C}$ with $\operatorname{Re}(w)>0$ and fix $x>0$ in $\mathbb{R}$, 
a straightforward calculation then yields
\begin{align*}
    \left|\frac{w}{w + x}\right| \leq \frac{|w|}{\sqrt{|w|^2 + x^2}}.
\end{align*}
The function on the right-hand side is increasing with respect to $|w|$.
Denote $\sigma_{\min}(M)$ as the smallest singular of $M$. Applying the interlacing inequalities for singular
values in Theorem 2 of \cite{thompson1972principal} gives
\begin{align*}
\max_j |\lambda_j(H_k)| \leq \sigma_{\max}(H_k)\leq \sigma_{\max}(M),
\qquad
\min_j |\lambda_j(H_k)| \geq \sigma_{\min}(H_k)\geq \sigma_{\min}(M).
\end{align*}
Letting $\mathcal{S}:=[\sigma_{\min}(M),\sigma_{\max}(M)]$, one arrives at
\begin{align*}
    \int_0^{\infty}x^{1 / 2} \prod_{j=1}^k\left|\frac{\lambda_j(H_k)}{\lambda_j(H_k)+x}\right|\, \mathrm{d}x&\leq \int_0^{\infty}x^{1 / 2} \prod_{j=1}^k\frac{|\lambda_j(H_k)|}{\sqrt{|\lambda_j(H_k)|^2 + x^2}}\, \mathrm{d}x\\
    &\leq \int_0^{\infty}x^{1 / 2} \left(\sup\limits_{|w|\in\mathcal{S}}\frac{|w|}{\sqrt{|w|^2 + x^2}}\right)^k\, \mathrm{d}x\\
    &\leq \int_0^{\infty}x^{1 / 2} \frac{(\sigma_{\max}(M))^{k}}{(\sqrt{(\sigma_{\max}(M))^2+x^2})^{k}}\, \mathrm{d}x.
\end{align*}
Letting $x=\sigma_{\max}(M)\tan(\theta)$, for $k\geq 2$, we obtain
\begin{align*}
\int_0^{\infty}
x^{1 / 2}
\frac{(\sigma_{\max}(M))^{k}}{\bigl(\sqrt{(\sigma_{\max}(M))^2+x^2}\bigr)^{k}}\, \mathrm{d}x
&= (\sigma_{\max}(M))^{3/2}
\int_0^{\frac{\pi}{2}}
\frac{\tan^{1/2}(\theta)\,\sec^2(\theta)}{\sec^{k}(\theta)}\,\mathrm{d}\theta \\
&= (\sigma_{\max}(M))^{3/2}
\int_0^{\frac{\pi}{2}}
\sin^{1/2}(\theta)\,\cos^{k-5/2}(\theta)\,\mathrm{d}\theta \\
&= \frac{(\sigma_{\max}(M))^{3/2}}{2}\,
\mathrm{B}\!\left(
\frac{2k-3}{4},\,
\frac{3}{4}
\right) \\
&= \frac{(\sigma_{\max}(M))^{3/2}}{2}\,
\frac{
\Gamma\!\left((2k-3)/4\right)\Gamma\!\left(3/4\right)}{
\Gamma\!\left(k/2\right)
},
\end{align*}
where $\mathrm{B}(\cdot, \cdot)$ is the Beta function, satisfying $\mathrm{B}(x,y) = \frac{\Gamma(x)\Gamma(y)}{\Gamma(x+y)}$.
Furthermore, following \cite{wendel1948note}, 
for $w>0$ and $l\in(0,1)$,   
\begin{align*}
    \frac{w}{(w+l)^{1-l}}\leq\frac{\Gamma(w+l)}{\Gamma(w)}\leq w^l.
\end{align*}
Then, with $w = (2k-3)/4$ and $l = 3/4$, we have,
\begin{align*}
\frac{\Gamma\!\left((2k-3)/4\right)}
{\Gamma\!\left(k/2\right)}
\le
\frac{4}{2k-3}
\left(\frac{k}{2}\right)^{1/4}=\frac{2k}{2k-3}\left(\frac{k}{2}\right)^{-3/4}.
\end{align*}
Collecting all estimates yields the desired result.
\end{proof}
\begin{remark}\label{tight_bounds}
Summarizing the above results, we obtain the following estimates:
\begin{align}
\left\|M^{1 / 2} \bm{b}-\operatorname{Arn}_k(x^{1 / 2} ; M, \bm{b})\right\|&\leq\frac{1}{\pi} \left(\int_0^{\infty}x^{1 / 2} \prod_{j=1}^k\left|\frac{\lambda_j(H_k)}{\lambda_j(H_k)+x}\right| \, \mathrm{d}x\right)\left\|\bm{\xi}_0^k\right\|\label{boundM_1}\\
&\leq \frac{1}{\pi}\left(\int_0^{\infty}x^{1 / 2} \prod_{j=1}^k\frac{|\lambda_j(H_k)|}{\sqrt{|\lambda_j(H_k)|^2 + x^2}}\, \mathrm{d}x\right)\left\|\bm{\xi}_0^k\right\|\label{bound_mod}\\
&\leq\frac{\Gamma\!(3/4)}{2^{1/4}\pi}\frac{2k}{2k-3}(\sigma_{\max}(M))^{\frac{3}{2}}k^{-\frac{3}{4}}\left\|\bm{\xi}_0^k\right\|.\label{bound_gamma}
\end{align}
The inequalities \cref{boundM_1} and \cref{bound_mod} provide a posteriori error bounds and remain relatively tight. The last inequality \cref{bound_gamma} gives an a priori bound; since it relies on  $|\lambda_j(H_k)| \le \sigma_{\max}(M)$, it becomes loose when the eigenvalues of $M$ are widely distributed. 
\end{remark}
\begin{remark}\label{Arn_spd}
If the matrix $M$ is Hermitian, then using the results from \cite{chen2022error} and \cite{amsel2024nearly}, we obtain
\begin{align}
&\left\|M^{1 / 2} \bm{b}-\operatorname{Arn}_k(x^{1 / 2} ; M, \bm{b})\right\|\nonumber\\
&\leq\frac{1}{\pi} \left(\int_0^{\infty}x^{1 / 2} \prod_{j=1}^k\frac{\lambda_j(H_k)}{\lambda_j(H_k)+x}\frac{\lambda_{\max}(M)}{\lambda_{\max}(M)+x}  \, \mathrm{d}x\right)\left\|\bm{\xi}_0^k\right\|\nonumber\\
&\leq\frac{1}{\pi} \left(\int_0^{\infty}x^{1 / 2} \left(\frac{\lambda_{\max}(M)}{\lambda_{\max}(M)+x}\right)^{k+1}  \, \mathrm{d}x\right)\left\|\bm{\xi}_0^k\right\|\nonumber\\
&=\frac{(\lambda_{k+1}(M))^{3/2}}{\pi}\mathrm{B}\!\left(
\frac{3}{2},\,
k-\frac{1}{2}\right)
\left\|\bm{\xi}_0^k\right\|\nonumber\\
&=\frac{(\lambda_{k+1}(M))^{3/2}}{\pi}\frac{\Gamma\!\left(3/2\right)\Gamma\!\left(k-1/2\right)}{
\Gamma\!\left(k+1\right)
}\left\|\bm{\xi}_0^k\right\|\nonumber\\
&\leq \frac{1}{2\sqrt{\pi}}\frac{2k}{2k-1}(\lambda_{\max}(M))^{\frac{3}{2}}k^{-\frac{3}{2}}\left\|\bm{\xi}_0^k\right\|\label{sp_loose_bound}.
\end{align} 
This sharper result stems from the fact that the Hermitian matrix is unitarily diagonalizable with real eigenvalues, which allows for a more refined spectral analysis of $\left|\frac{\det(H_k)}{\det(H_k-zI)}\right|$ and $\|(M+xI)^{-1}M\|$.  
\end{remark}

Thus, as a byproduct of our analysis, we derive an error bound for the Hermitian case that is tighter than \eqref{sp_loose_bound}, the result from \cite{chen2022error}.  

\begin{theorem}\label{spd_shper}
    Given a positive-definite Hermitian matrix $M \in \mathbb{C}^{n \times n}$ and a vector $\bm{b} \in \mathbb{C}^n$, the error of the Arnoldi process for approximating the matrix square root is bounded by
\begin{align}
\left\|M^{1 / 2} \bm{b}-\operatorname{Arn}_k(x^{1 / 2} ; M, \bm{b})\right\|\leq \frac{1}{2\sqrt{\pi}}\frac{2k}{2k-1}\left(\bar{\lambda}_{k+1}(M)\right)^{\frac{3}{2}}k^{-\frac{3}{2}}\left\|\bm{\xi}_0^k\right\|\label{sp_Jensen},
\end{align}
where
\begin{align*}
    \bar{\lambda}_{k+1}(M):=\frac{1}{k+1}\left(\sum_{j=1}^k\lambda_j(M)+\lambda_1(M)\right) \leq \lambda_{\max}(M),
\end{align*}
with $\lambda_n(M)\leq\cdots\leq\lambda_2(M)\leq\lambda_1(M):=\lambda_{\max}(M)$ denoting the eigenvalues of $M$. 
\end{theorem}
\begin{proof}
Using the arithmetic–geometric mean inequality, we obtain
\begin{align*}
\prod_{j=1}^k \frac{\lambda_j(H_k)}{\lambda_j(H_k) + x}\frac{\lambda_1(M)}{\lambda_1(M)+x}
\leq
\left(\frac{1}{k+1}\left(
\sum_{j=1}^{k}
\frac{\lambda_j(H_k)}{\lambda_j(H_k) + x}
+\frac{\lambda_1(M)}{\lambda_1(M) + x}\right)\right)^{k+1}.
\end{align*}
Observing that $g_x(\lambda):=\frac{\lambda}{\lambda+x}$ is a concave function in $\lambda$, Jensen’s inequality yields
\begin{align*}
\frac{1}{k+1}\left(
\sum_{j=1}^{k}
g_x(\lambda_j(H_k))
+g_x(\lambda_1(M))\right)\leq g_x\left(\frac{\sum_{j=1}^k\lambda_j(H_k)+\lambda_1(M)}{k+1}\right).
\end{align*}
Then, 
\begin{align*}
&\left\|M^{1 / 2} \bm{b}-\operatorname{Arn}_k(x^{1 / 2} ; M, \bm{b})\right\|\\
&\leq\frac{1}{\pi} \left(\int_0^{\infty}x^{1 / 2} \prod_{j=1}^k\frac{\lambda_j(H_k)}{\lambda_j(H_k)+x} \frac{\lambda_1(M)}{\lambda_1(M) + x}\, \mathrm{d}x\right)\left\|\bm{\xi}_0^k\right\|\\
&\leq \frac{1}{\pi} \left(\int_0^{\infty}x^{1 / 2} \left(\frac{\frac{1}{k+1}\left(\sum_{j=1}^k\lambda_j(H_k)+\lambda_1(M)\right)}{\frac{1}{k+1}\left(\sum_{j=1}^k\lambda_j(H_k)+\lambda_1(M)\right)+x}\right)^{k+1} \, \mathrm{d}x\right)\left\|\bm{\xi}_0^k\right\|.
\end{align*}
Utilizing the interlacing inequality gives 
\begin{align*}
\left\|M^{1 / 2} \bm{b}-\operatorname{Arn}_k(x^{1 / 2} ; M, \bm{b})\right\|&\leq \frac{1}{\pi} \left(\int_0^{\infty}x^{1 / 2} \left(\frac{\bar{\lambda}_{k+1}(M)}{\bar{\lambda}_{k+1}(M)+x}\right)^{k+1} \, \mathrm{d}x\right)\left\|\bm{\xi}_0^k\right\|\nonumber\\
&\leq \frac{1}{2\sqrt{\pi}}\frac{2k}{2k-1}\left(\bar{\lambda}_{k+1}(M)\right)^{\frac{3}{2}}k^{-\frac{3}{2}}\left\|\bm{\xi}_0^k\right\|.
\end{align*}
This implies the bound \eqref{sp_Jensen} is sharper than the bound \eqref{sp_loose_bound} almost for all cases. 
\end{proof}

\begin{remark}
    Additionally, one can similarly derive an error estimate for the Arnoldi approximation of the matrix inverse square root. We briefly outline the key steps of the derivation. By arguments analogous to those used for \cref{error_square} and \cref{error_lo}, we obtain
\begin{align*}
    \left\|M^{-1 / 2} \bm{b}-\operatorname{Arn}_k(x^{-1 / 2} ; M, \bm{b})\right\|&\leq \frac{1}{\pi}\left(\int_0^{\infty}x^{-1 / 2} \frac{(\sigma_{\max}(M))^{k}}{(\sqrt{(\sigma_{\max}(M))^2+x^2})^{k}}\, \mathrm{d}x\right)\left\|\bm{\xi}_0^k\right\|\\
    &=\frac{(\sigma_{\max}(M))^{1/2}} {\pi}
\left(\int_0^{\frac{\pi}{2}}
\sin^{-1/2}(\theta)\cos^{k-3/2}(\theta)\,\mathrm{d}\theta\right)\left\|\bm{\xi}_0^k\right\|\\
&=\frac{(\sigma_{\max}(M))^{1/2}} {2\pi}\mathrm{B}\!\left(
\frac{2k-1}{4},\,
\frac{1}{4}\right)\left\|\bm{\xi}_0^k\right\|\\
&=\frac{(\sigma_{\max}(M))^{1/2}}{2\pi}\,
\frac{
\Gamma\!\left((2k-1)/4\right)\Gamma\!\left(1/4\right)}{
\Gamma\!\left(k/2\right)
}\left\|\bm{\xi}_0^k\right\|\\
&\leq \frac{\Gamma(1/4)}{ 2^{3/4}\pi}\frac{2k}{2k-1}\left(\sigma_{\max}(M)\right)^{\frac{1}{2}}k^{-\frac{1}{4}}\left\|\bm{\xi}_0^k\right\|.
\end{align*}
If $M$ is Hermitian, then following \cite{amsel2024nearly} and the discussion above, we can derive the following tighter bound, which also improves upon the state-of-the-art result,
\begin{align*}
\left\|M^{-1 / 2} \bm{b}-\operatorname{Arn}_k(x^{-1 / 2} ; M, \bm{b})\right\|&\leq\frac{1}{\pi} \left(\int_0^{\infty}x^{-1 / 2} \left(\frac{\bar{\lambda}_{k+1}(M)}{\bar{\lambda}_{k+1}(M)+x}\right)^{k+1} \, \mathrm{d}x\right)\left\|\bm{\xi}_0^k\right\|\\
&=\frac{(\bar{\lambda}_{k+1}(M))^{1/2}}{\pi}\mathrm{B}\!\left(
\frac{1}{2},\,k+\frac{1}{2}
\right)
\left\|\bm{\xi}_0^k\right\|\\
&=\frac{(\bar{\lambda}_{k+1}(M))^{1/2}}{\pi}\frac{\Gamma\!\left(1/2\right)\Gamma\!\left(k+1/2\right)}{
\Gamma\!\left(k+1\right)
}\left\|\bm{\xi}_0^k\right\|\\
&\leq \frac{1}{\sqrt{\pi}}\frac{2k+2}{2k+1}\left(\bar{\lambda}_{k+1}(M)\right)^{\frac{1}{2}}(k+1)^{-\frac{1}{2}}\left\|\bm{\xi}_0^k\right\|.
\end{align*}
\end{remark}


\subsection{Analysis for a Perturbed Matrix} In practice, matrices arising from large-scale problems are often too large to be directly measured, stored, or analyzed, necessitating efficient preprocessing. 
To further reduce the cost, one direction is to employ data-sparse approximations, such as hierarchical or low-rank representations \cite{bjarkason2019pass,clarkson2017low,datta2010numerical,hackbusch2015hierarchical,kishore2017literature,lin2011fast}. 
These methods accelerate matrix-vector multiplications by compressing the underlying structure of the matrix. Thus, we generalize the results from \cref{error_square} to get an error bound for approximating the square root of a perturbed matrix using the Arnoldi iteration.


\begin{corollary}\label{error_square_real}
Let $M \in \mathbb{C}^{n \times n}$ be a positive-definite matrix such that for any unit vector $\bm{x}$, $\operatorname{Re}(\bm{x}^*M\bm{x}) \geq \mu_1^2$ for some $\mu_1>0$, and let $\bm{b} \in \mathbb{C}^n$ be a given vector. We consider a perturbed matrix $\widetilde{M}$, satisfying $\operatorname{Re}(\bm{x}^*\widetilde{M}\bm{x}) \geq \mu_2^2$ for some $\mu_2>0$, as well as 
$\|M - \widetilde{M}\| \leq \epsilon \|M\|$ for some $\epsilon>0$. Then, the error of the Arnoldi process for approximating the matrix square root is bounded by
\begin{align}\label{bound_prepo}
&\left\|M^{1/2}\bm{b}-\operatorname{Arn}_k(x^{1 / 2} ; \widetilde{M}, \bm{b})\right\|\nonumber\\
&\leq 
\frac{\epsilon\sigma_{\max}(M)}{\mu_1+\mu_2}\left\|\bm{b}\right\|+ \frac{\Gamma\!(3/4)(1+\epsilon)^{3/2}}{2^{1/4}\pi}\frac{2k}{2k-3}(\sigma_{\max}(M))^{\frac{3}{2}}k^{-\frac{3}{4}}\left\|\bm{\xi}_0^k\right\|.
\end{align}
\end{corollary}

\begin{proof}
Based on $\widetilde{M}$, we  apply the Arnoldi process to compute the matrix square root. Now the problem becomes
\begin{align*}
    \left\|M^{1 / 2} \bm{b}-\operatorname{Arn}_k(x^{1 / 2} ; \widetilde{M}, \bm{b})\right\|\leq \left\|M^{1 / 2} \bm{b}- \widetilde{M}^{1/2}\bm{b}\right\|+\left\|\widetilde{M}^{1/2}\bm{b}-\operatorname{Arn}_k(x^{1 / 2} ; \widetilde{M}, \bm{b})\right\|.
\end{align*}
The first term can be bounded using the perturbation estimate in \cite{schmitt1992perturbation}, thus,
\begin{align*}
    \left\|M^{1 / 2} \bm{b}- \widetilde{M}^{1/2}\bm{b}\right\|\leq \left\|M^{1 / 2}- \widetilde{M}^{1/2}\right\|\|\bm{b}\|\leq \frac{1}{\mu_1+\mu_2}\left\|M-\widetilde{M}\right\|\|\bm{b}\|\leq\frac{\epsilon\sigma_{\max}(M)}{\mu_1+\mu_2}\left\|\bm{b}\right\|.
\end{align*}
By \cref{error_lo}, the second term is bounded by:
\begin{align*}
   \left\|\widetilde{M}^{1/2}\bm{b}-\operatorname{Arn}_k(x^{1 / 2} ; \widetilde{M}, \bm{b})\right\|\leq \frac{\Gamma\!(3/4)}{2^{1/4}\pi}\frac{2k}{2k-3}(\sigma_{\max}(\widetilde{M}))^{\frac{3}{2}}k^{-\frac{3}{4}}\left\|\bm{\xi}_0^k\right\|.
\end{align*}
By Weyl's inequality \cite{bhatia1997matrix}, we have $|\sigma_k(M)-\sigma_k(\widetilde{M})| \leq \left\|M-\widetilde{M}\right\|\leq \epsilon\left\|M\right\|=\epsilon\sigma_{\max}(M)$. Combining the above estimates yields  \eqref{bound_prepo}
\end{proof}

\section{Numerical experiments and discussion} \label{numerical}In this section, we report on numerical studies of the Arnoldi iteration for computing the matrix square root applied to a vector, that is, $M^{1/2}\bm{b}$. 
We first evaluate the a \textit{posteriori} bounds \eqref{boundM_1} and \eqref{bound_mod}, and the a \textit{priori} bound \eqref{bound_gamma} using non-Hermitian matrices. Next, we compare the a \textit{priori} bounds \eqref{sp_loose_bound} and \eqref{sp_Jensen} in the Hermitian setting to demonstrate the improvement yielded by our new bound \eqref{sp_Jensen}. Then, we investigate the dependence of the approximation error on the largest singular value $\sigma_{\max}(M)$ and the iteration count $k$. Finally, we explore the use of the Arnoldi iteration for computing $M^{1/2}\bm{b}$ in the simulation of particulate suspensions to demonstrate the effectiveness of the proposed method in practical applications. 
In particular, all integrals in the bounds are evaluated numerically using \texttt{integral} in MATLAB with a relative tolerance of $10^{-8}$ and an absolute tolerance of $10^{-12}$.

\subsection{Error bounds}\label{dis_bound}
Here, we compare the error bound estimates for cases when the spectrum of the matrix is uniformly distributed versus when the eigenvalues are clustered.  We test this on both non-Hermitian and Hermitian examples.

\subsubsection{The non-Hermitian case}
Consider a positive definite non-Hermitian matrix $M\in\mathbb{R}^{500\times 500}$, constructed as  
$M = M_0 + K$. 
The symmetric component $M_0$ is defined using $M_0 = Q \mathrm{diag}(\Lambda) Q^T$, 
where $Q$ is an orthogonal matrix and $\Lambda$ is a diagonal matrix. 
The matrix $K$ is introduced as a skew-symmetric term derived from a general $n \times n$ random matrix $R$ with $\mathcal{N}(0, 1)$ entries: 
$$K = \tfrac{1}{2}(R - R^T).$$

First, we construct $\Lambda$ such that its eigenvalues are uniformly distributed in the interval $[1,1000]$. \cref{wide_M} displays the eigenvalues in the complex plane and the distribution of their modulus. We see that the eigenvalues of $M$ are widely distributed. 
\begin{figure}[htbp]
    \centering
    \begin{minipage}[t]{0.5\textwidth}
        \centering
        \includegraphics[width=\textwidth]{ 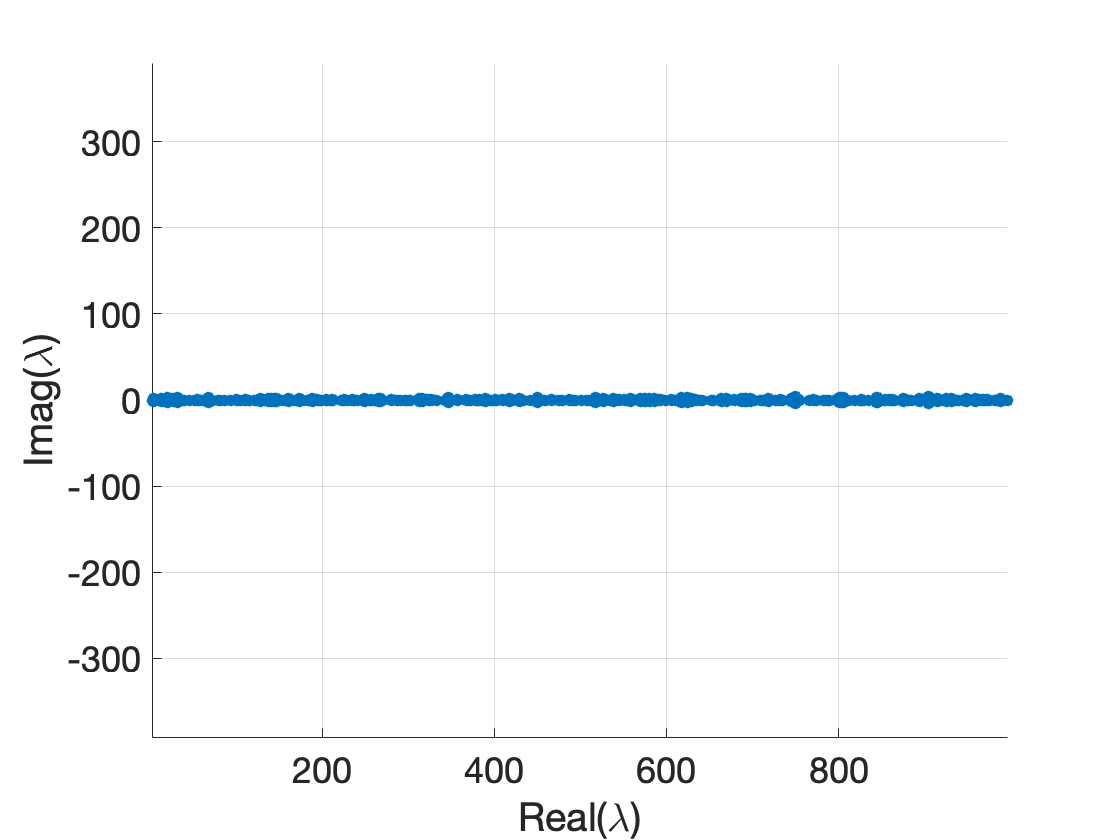}
    \end{minipage}%
    \hfill
    \begin{minipage}[t]{0.5\textwidth}
        \centering
        \includegraphics[width=\textwidth]{ 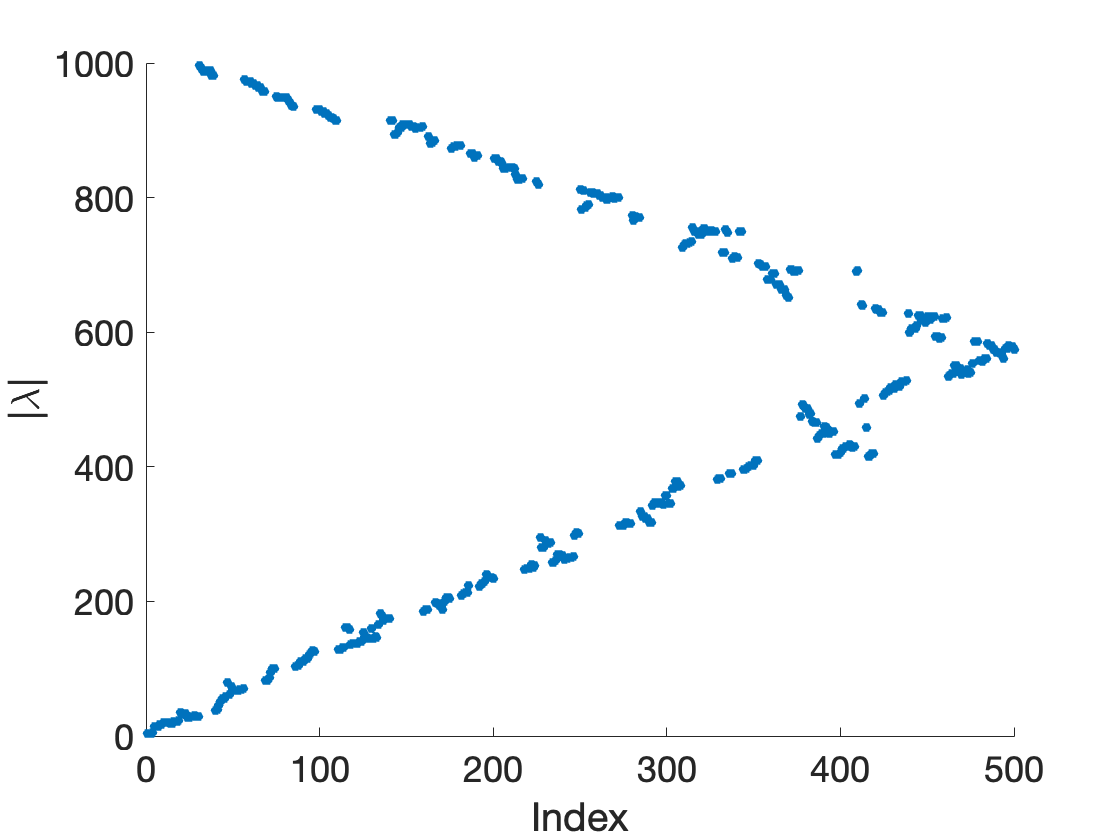}
    \end{minipage}
    \caption{Eigenvalues of the uniformly-distributed spectrum matrix in the complex plane (left) and their magnitudes (right).}\label{wide_M}
\end{figure}

Setting $\bm{b}$ as an all-ones vector, the error of the Arnoldi process  for computing $M^{1/2}\bm{b}$  is shown in \cref{wide_M_error}. We observe that the intermediate bound \cref{bound_mod} remains relatively sharp. The overestimation in the a priori bound \cref{bound_gamma} mainly arises from the rough estimate $|\lambda_j(H_k)| \le \sigma_{\max}(M)$, which is consistent with our earlier discussion in \cref{tight_bounds}. 

\begin{figure}[h]
    \centering    \includegraphics[width=0.8\textwidth]{ 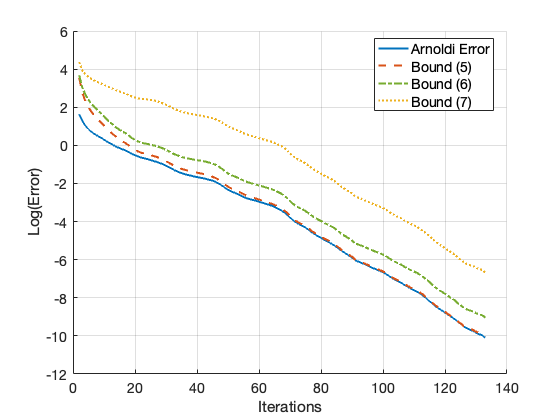}
    \caption{Error of the Arnoldi process for a non-Hermitian matrix with a uniformly-distributed spectrum.}\label{wide_M_error}
\end{figure}

Next, let $\Lambda = [\lambda_{\text{cluster}}; \lambda_{\text{outlier}}]$, where 95\% of the eigenvalues are normally distributed around 1000:$$\lambda_{\text{cluster}} \sim 1000 + 100\,\mathcal{N}(0,1).$$
The remaining outliers are drawn from a smaller normal distribution centered near 10:
$$\lambda_{\text{outlier}} \sim 10 + 5\,\mathcal{N}(0,1).$$
The distribution of eigenvalues and the magnitudes $|\lambda_j(M)|$ are displayed in \cref{cluster_M}. As shown, most of the $|\lambda_j(M)|$ lie around 1000, with only a few small isolated values. 

\begin{figure}[htbp]
    \centering
    \begin{minipage}[t]{0.5\textwidth}
        \centering
        \includegraphics[width=\textwidth]{ 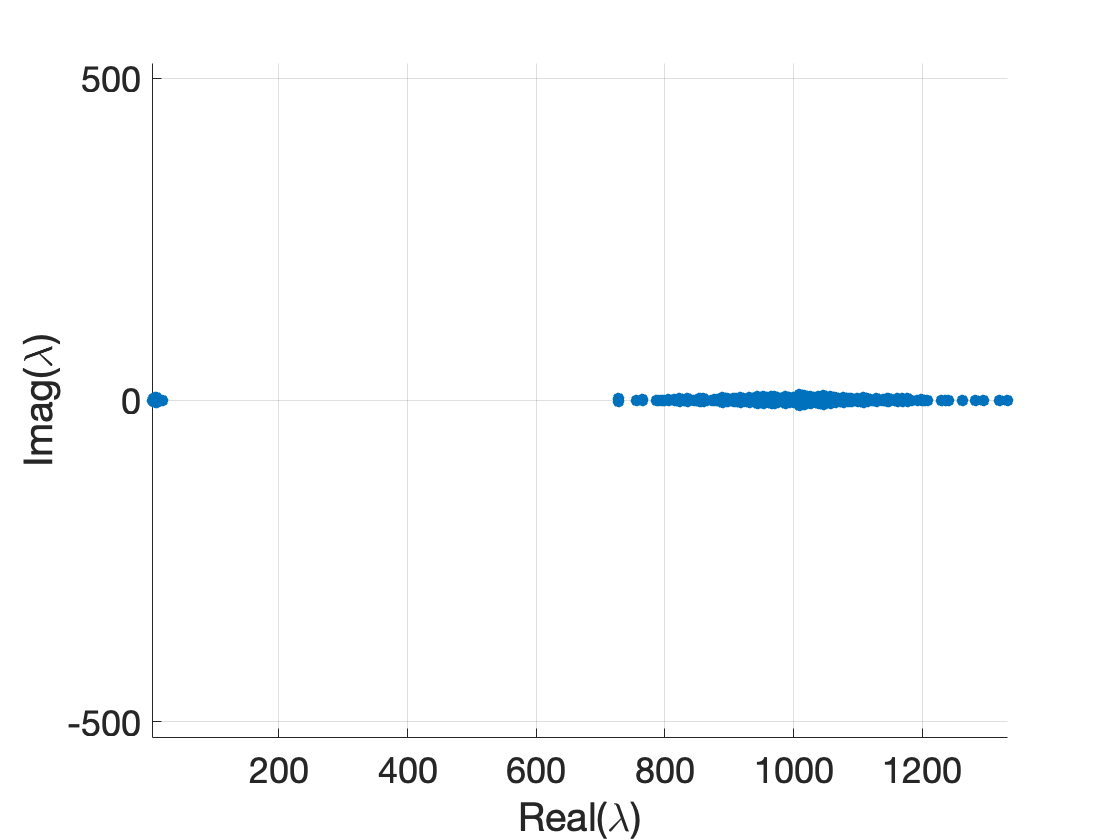}
    \end{minipage}%
    \hfill
    \begin{minipage}[t]{0.5\textwidth}
        \centering
        \includegraphics[width=\textwidth]{ 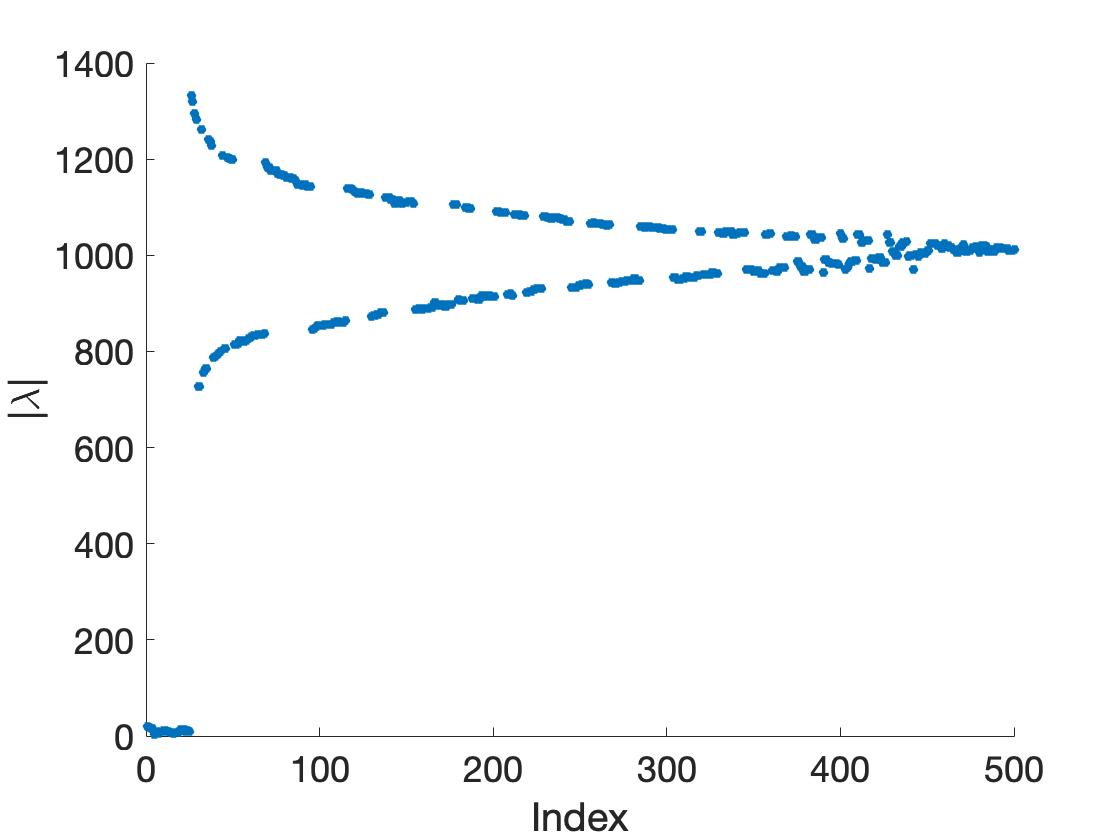}
    \end{minipage}
    \caption{Eigenvalues of the clustered spectrum matrix in the complex plane (left) and their magnitudes (right).}\label{cluster_M}
\end{figure}

The vector $\bm{b}$ is taken to be the average of 100 eigenvectors corresponding to the 100 eigenvalues of largest magnitude. Then, the corresponding results of the Arnoldi process 
are presented in \cref{cluster_M_error}. Due to the clustered eigenvalue distribution, the method exhibits fast convergence, with a priori bounds \eqref{bound_gamma} providing a sharper estimate in this regime as expected.  

\begin{figure}[h]
    \centering    \includegraphics[width=0.8\textwidth]{ 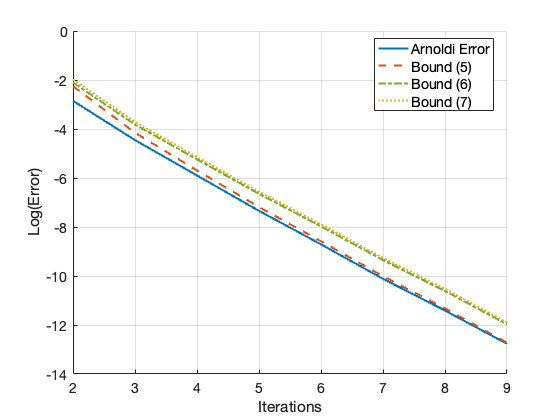}
    \caption{Error of the Arnoldi process for a non-Hermitian matrix with a clustered spectrum.}\label{cluster_M_error}
\end{figure}

\subsubsection{The Hermitian case} Next, we construct a $500\times500$ Hermitian matrix as
$M = Q\,\mathrm{diag}(\Lambda)\,Q^T$, where the eigenvalues of the diagonal matrix $\Lambda$ follow a uniform distribution over $[1, 1000]$. $\bm{b}$ is an all-ones vector.
In  \cref{cluster_M_error_spd_uni}, we report the error and bounds in \cref{Arn_spd} and \cref{spd_shper}. We observe that, as $k$ increases, the Ritz value $\lambda_j(H_k)$ approaches a true eigenvalue of $M$, leading the bound \eqref{boundM_1} to become sharper and align more closely with the actual error. We also plot the evolution of $\bar{\lambda}_{k+1}(M)$, with $\sigma_{\max}(M)$ shown in \cref{spec_com_uni} for comparison. It can be observed that when the eigenvalues are uniformly distributed and lie close to one another, bound \eqref{sp_Jensen}  yields only a minimal improvement over bound \eqref{sp_loose_bound}.

\begin{figure}[h]
    \centering    \includegraphics[width=0.8\textwidth]{ 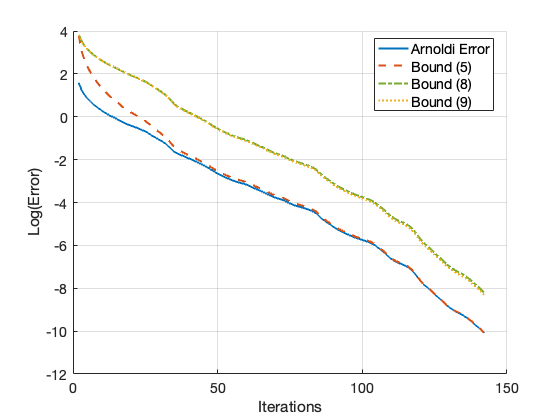}
    \caption{Error of the Arnoldi process for a Hermitian matrix with a uniformly-distributed spectrum.}\label{cluster_M_error_spd_uni}
\end{figure}

\begin{figure}[h]
    \centering    \includegraphics[width=0.8\textwidth]{ 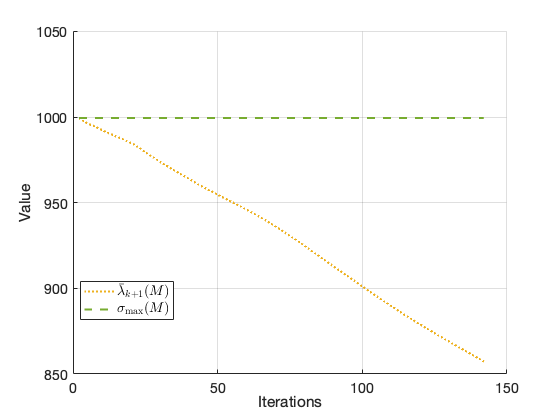}
    \caption{Evolution of $\bar{\lambda}_{k+1}$ for a symmetric matrix with a uniformly-distributed spectrum.}\label{spec_com_uni}
\end{figure}

Next, we take $\Lambda = [\lambda_{\text{cluster}};\, \lambda_{\text{outlier}}]$ with 99\% of the eigenvalues tightly clustered near 10, 
and the remaining 1\% sampled as outliers  around 1000. The same $\bm{b}$ is used. The results are displayed in \cref{cluster_M_error_spd} and \cref{spec_com}. 
Clearly, as $k$ increases, the sequence $\bar{\lambda}_{k+1}(M)$ decreases monotonically, and becomes progressively smaller than $\sigma_{\max}(M)$. Therefore, as expected in \cref{tight_bounds}, when the larger eigenvalues are more widely spread, bound \eqref{sp_Jensen} exhibits a tighter estimate compared to bound \eqref{sp_loose_bound}.

\begin{figure}[h]
    \centering    \includegraphics[width=0.8\textwidth]{ 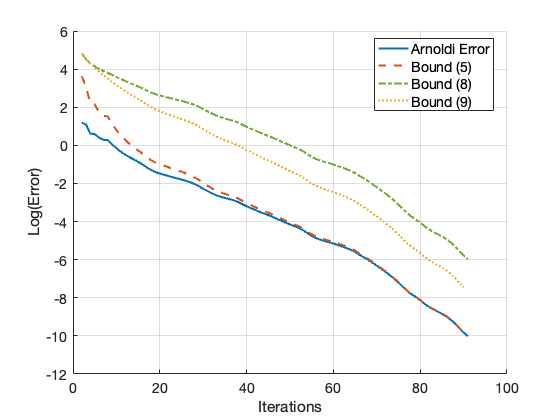}
    \caption{Error of the Arnoldi process for a symmetric matrix with a clustered spectrum.}\label{cluster_M_error_spd}
\end{figure}

\begin{figure}[h]
    \centering    \includegraphics[width=0.8\textwidth]{ 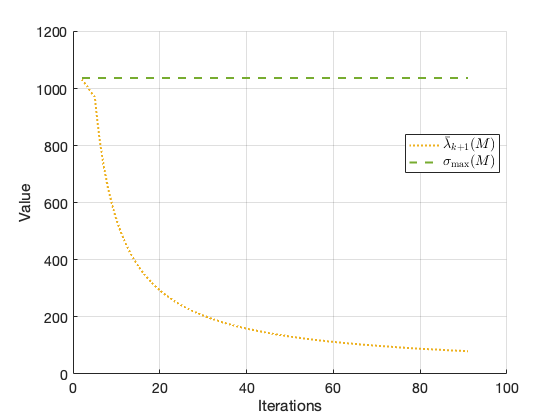}
    \caption{Evolution of $\bar{\lambda}_{k+1}$ for a symmetric matrix with a clustered spectrum.}\label{spec_com}
\end{figure}

\subsection{Scaling behavior} 

In this example, we test the algorithm's dependence on the number of Arnoldi iterations $k$ and the largest singular value $\sigma_{\max}(M)$. 
Here, we construct $M$ from the discretization of a convection-diffusion problem. The continuous convection-diffusion operator is defined as
$$Lu:=-\eta u^{{\prime}{\prime}}(x) + u^{\prime}(x), \quad x \in [0,1],$$
subject to homogeneous Dirichlet boundary conditions. Here, $\eta$ is the diffusion coefficient. For small values of $\eta$, the problem is convection-dominated. To accurately capture the convection effects while maintaining numerical stability, we employ an upwind scheme. A uniform numerical grid with grid spacing $h = \frac{1}{n-1}$ is adopted, where the discrete points are defined as $x_i = ih$, $i = 0, 1,2,\dots, n-1$. The resulting finite-difference scheme of the operator takes the following form:
$$ -\eta \frac{u_{i+1} - 2u_i + u_{i-1}}{h^2} + \frac{u_i - u_{i-1}}{h},$$
where $u_{i}$ represents the numerical approximation to $u(x_i)$. Let $M$ be the tridiagonal matrix obtained from the above finite-difference discretization of $L$, which is non-Hermitian. 


Here, we take $\bm{b}$ to be the vector with all entries equal to one, and choose $\eta$ = 0.1 and $n=500$. In this case, $M$ exhibits a relatively dispersed eigenvalue distribution. We compare different error bounds for approximating the square root of $M$ using the Arnoldi process described above (see \cref{fig:bound2}).  Similar phenomena as in \Cref{dis_bound} can also be seen here. Although the bound \cref{bound_gamma} is loose, its overall trend matches that of the true error, providing insights into the inherent scaling behavior.

\begin{figure}[h]
    \centering       \includegraphics[width=0.8\textwidth]{ 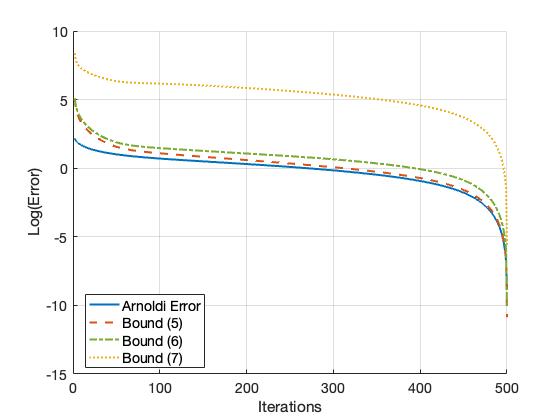}
    \caption{Error of the Arnoldi process for a  convection-diffusion problem with $\eta=0.1$ and $n=500$.}
    \label{fig:bound2}
\end{figure}

Next, we examine how the error depends on $\sigma_{\max}(M)$  and $k$, and whether the observed behavior follows the scaling predicted by bound \eqref{bound_gamma}. To this end,
 we introduce the scaled error derived from   \eqref{bound_gamma}, defined as
\begin{align}\label{scaled error}
    \frac{2^{1/4}\pi}{\Gamma\!(3/4)}\frac{2k-3}{2k}\frac{\left\|M^{1 / 2}\bm{b}-\operatorname{Arn}_k(f ; M, \bm{b})\right\|}{\left\|\bm{\xi}_0^k\right\|}.
\end{align}
In the following tests, we use the stopping criterion in \cref{stop} with \textit{tol} = 0.05.

The results of approximating $M^{1/2}\bm{b}$ for matrices of varying dimensions arising from the convection-diffusion problem are summarized in  \cref{cv}. We see that as $n$ increases, maintaining the same accuracy requires a larger number of iterations. The log-log plot of the scaling term decaying against the iteration count $k$ is shown in \cref{cv_k}. The convergence rate exhibits a decay rate of approximately $\mathcal{O}(k^{-3/4})$ with respect to the iteration count $k$. 
This observation is consistent with the theoretical result established in \cref{error_lo}. Additionally, we fix $k=300,500,700$ and plot the scaling term against the largest singular value on a log-log scale in \cref{cv_cd}. The results show that the fitted slopes vary with $k$, but, as expected, they all fall below the theoretical prediction.

\begin{table}[h]\label{cv}
\renewcommand{\arraystretch}{1.1}
\centering
\begin{center}
\caption{Largest singular value of $M$ and  iteration counts $n$ and errors of the Arnoldi process for a convection-diffusion problem.}  
\begin{tabular}{c|c|c|c}
\hline $n$&$\sigma_{\max}(M)$ & error & Iteration Counts $k$ \\
\hline 
1000 & 202320.64  & 0.03083  & 889  \\
1200 & 291138.58  & 0.03053  & 1071 \\
1400 & 396074.49  & 0.03061  & 1253 \\
1600 & 517128.36  & 0.03090  & 1435 \\
1800 & 654300.20  & 0.03132  & 1617 \\
2000 & 807590.00  & 0.03132  & 1800 \\
\hline
\end{tabular}
\end{center}
\end{table}


\begin{figure}[h]
    \centering    \includegraphics[width=0.8\textwidth]{ 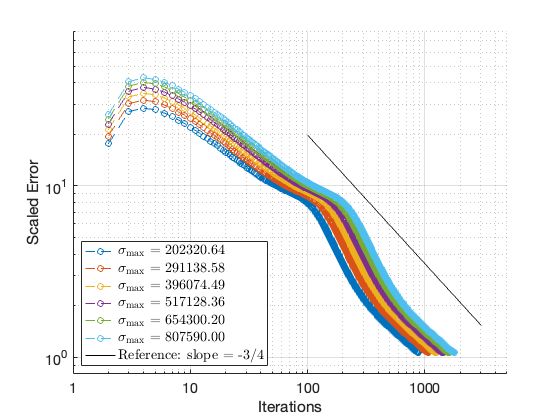} 
    \caption{Log-log plot of the scaled error decay with varying condition numbers against iteration count $k$  for a convection-diffusion problem.}\label{cv_k}
\end{figure}

\begin{figure}[h]
    \centering    \includegraphics[width=0.8\textwidth]{ 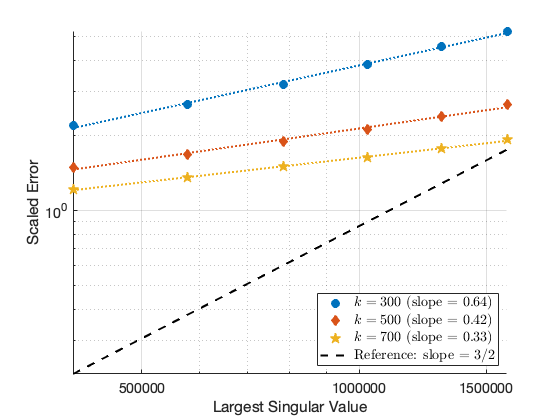} 
    \caption{Log-log plot of the scaled error with different $k$ against the largest singular value for a convection-diffusion problem.}\label{cv_cd}
\end{figure}

\subsection{Practical application} The study of particulate suspensions is of great interest, and plays critical roles in diverse fields such as materials design \cite{makey2020universality}, biological systems \cite{makey2020universality}, micro-robotics \cite{yang2020reconfigurable}, and food processing \cite{kim2020modeling}. Immersed in a fluid environment, particles' mobility is dictated by the hydrodynamic interactions between the particles. At micro scales, particles are also subject to thermal fluctuations, in the form of Brownian motion. 
The configuration of particles evolves over time and is governed by 
the overdamped limit of Langevin dynamics \cite{ermak1978brownian}. Thus, the computation of particles' Brownian displacements crucially rely on the mobility matrix, and its square root, both dependent on all particles' instantaneous positions.  Efficiently computing these two matrices remains a significant computational challenge. 

Since hydrodynamic interactions depend on the separation distance between particles, the particle interactions weaken as the distance increases \cite{ma2022fast}. As a result, far-field blocks in the mobility matrix exhibit low-rank structure, and can then be efficiently approximated.
To this end, the Hydrodynamic Interaction Graph Neural Network (HIGNN) framework,  proposed in \cite{ma2022fast,ma2024shape}, maps the many-body  hydrodynamic interactions onto a graph neural network. By employing  hierarchical-matrices via adaptive cross approximation  \cite{hackbusch2015hierarchical}, the computational complexity for using HIGNN to compute particles' velocities given any forces can be reduced to nearly linear \cite{ma2025mathcal}.

In the following tests, $M$ is the mobility  matrix obtained from the HIGNN model and has been preprocessed into a hierarchical-matrix form. 
Also, we assume that a uniform gravitational force $\bm{b}$ is applied to each particle. This numerical example focuses on evaluating the effectiveness of the Arnoldi iteration on computing $M^{1/2}$ in a long-time dynamic simulation of particles' sedimentation. In the simulations, the explicit Euler method was adopted as the temporal integrator to update particles' positions from their velocities, with a time step of $\Delta t = 0.005$. At each time step, the HIGNN is applied to predict the velocities of all particles. 
To generate the initial cubic-lattice like configuration of particles, the particles are first positioned at grid points with a spacing of 4, and then each particle is randomly perturbed with a displacement of magnitude 0.2 from its original position. 
With this initial configuration, we simulate particles sedimenting in an unbounded domain. The number of particles, $N$, considered here varies from $20^3$ to $100^3$. For such large systems, direct computation of the square root of the mobility matrix $M$ can be computationally prohibitive. To address this challenge and also validate the accuracy of the Arnoldi approximation in this practical setting, we apply the Arnoldi process twice to approximate $M\bm{b}=M^{1/2}(M^{1/2}\bm{b})$. Thus, the relative error is defined as $\frac{\|M\bm{b} - \operatorname{Arn}_k(x^{1 / 2} ; M, \operatorname{Arn}_k(x^{1 / 2} ; M, \bm{b}))\|}{\|M\bm{b}\|}$. 

\begin{figure}[h!]
    \centering    \includegraphics[width=1\textwidth]{ 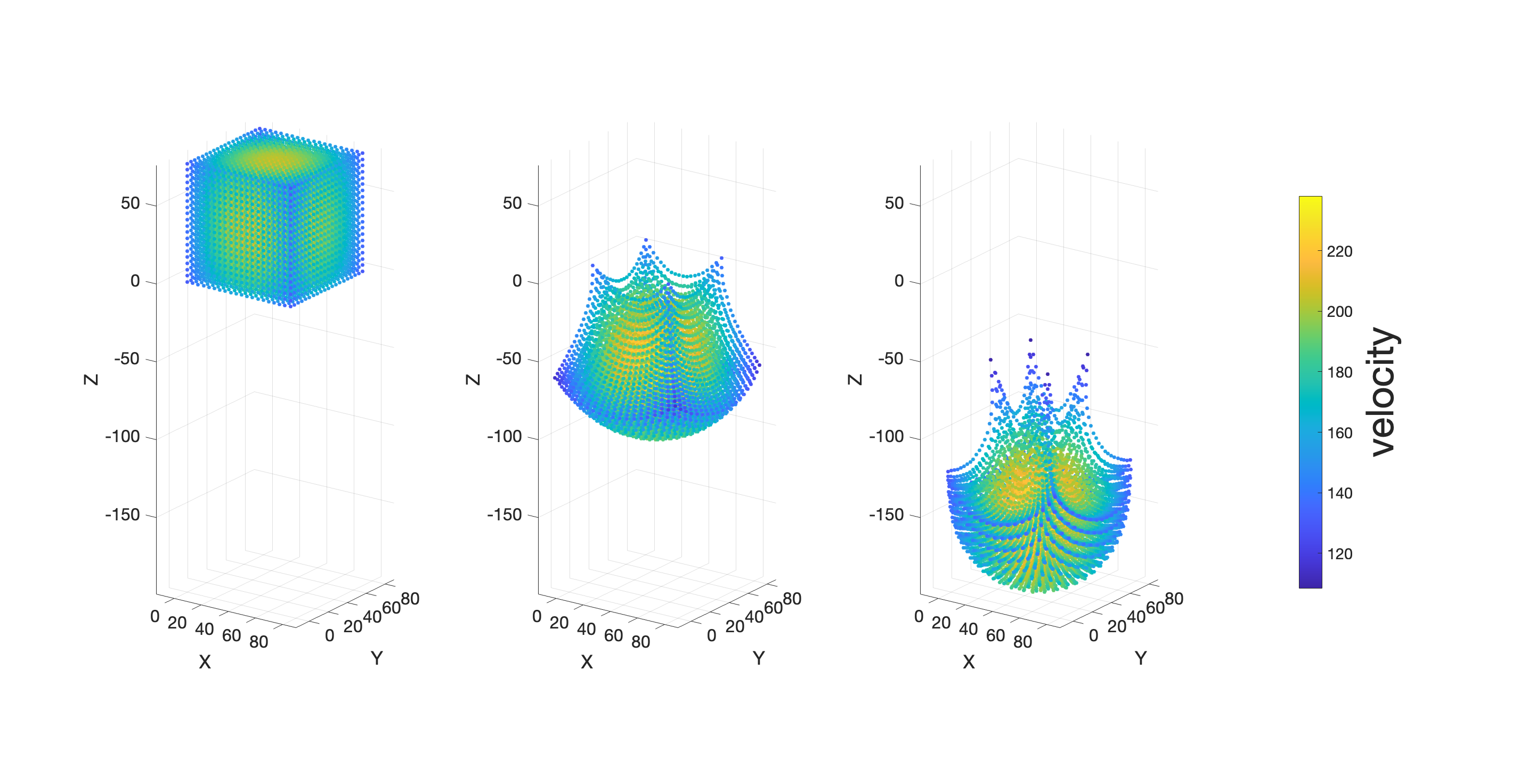} 
    \caption{Snapshots of 8000 particles sedimenting in an unbounded domain at times $t = 0.005$, $0.25$, and $1$, respectively, simulated using the HIGNN \cite{ma2025mathcal}. The color is correlated to the particles' velocity magnitude.}
    \label{snap} 
\end{figure}

\begin{table}[h!]
\renewcommand{\arraystretch}{1.1}
\centering
\begin{center}
\caption{Iteration counts and errors of the Arnoldi process for a particulate suspension problem with varying particle number $N$ at fixed concentration.} \label{lattice}
\label{Mhalf}
\begin{tabular}{c|c|c}
\hline$N$ 
&Iteration Counts $k$ & Relative Error \\
\hline 
$20^3$  &10 &2.06699e-04 \\
$40^3$  &14 &1.34295e-04 \\
$60^3$  &15 &1.70283e-04 \\
$80^3$  &18 &3.11680e-04 \\
$100^3$ &21 &3.86345e-04 \\
\hline
\end{tabular}
\end{center}
\end{table}

The snapshots at the 1st, 100th, and 200th time steps from the simulation of a $20^3$ particle system are presented in \cref{snap}. Letting $M$ be the mobility matrix at the 100th time step, \cref{lattice} records the number of iterations $k$ required to compute $M^{1/2}\bm{b}$ and the relative error as the number of particles $N$ increases from $20^3$ to $100^3$. Note that in this scenario, since the initial spacing between particles is kept the same as $N$ increases, a fixed particle concentration is maintained. 
One sees that as the number of particles increases, the number of iterations required to compute the matrix square root applied to a vector remains relatively low and does not grow significantly, while maintaining a relative error on the order of $10^{-4}$. This is because the matrix is well conditioned. Furthermore, we depict the relationship between the number of particles and the runtime on a log-log scale in \cref{log_log}. The plot implies that the Arnoldi approximation, based on the hierarchical matrix representation, exhibits an  $O(N \log N)$  time complexity.

\begin{figure}[h]
    \centering    \includegraphics[width=0.8\textwidth]{ 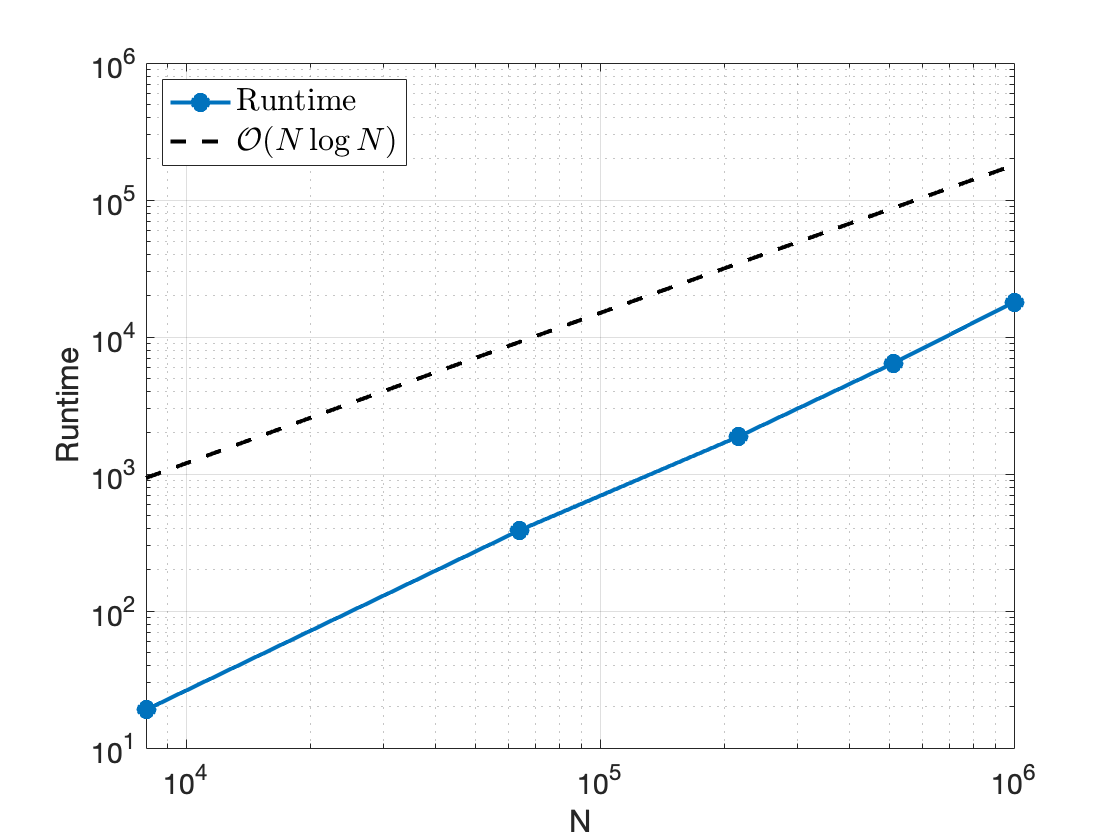} 
    \caption{Log-log plot of runtime versus number of particles $N$ at fixed particle concentration.}    
    \label{log_log} 
\end{figure}

Next, we consider $40^3$ particles arranged in a similar initial cubic lattice as above, but with the grid spacing $d$ varying from 3 to 5 (still perturbed by a displacement of 0.2). With this change, we can examine varying particle concentrations; as $d$ increases, the particle concentration decreases, thus lowering the required rank of the approximation to the corresponding blocks. 
As shown in \cref{sphere}, lower particle concentration leads to reduced runtime for the Arnoldi approximation. Also, the required number of iterations remains relatively small and stable compared to the matrix size. In this test, $M$ corresponds to the mobility matrix attained at the 50th time step.




\begin{table}[h]
\renewcommand{\arraystretch}{1.1}
\centering
\begin{center}
\caption{Iteration counts, errors, and runtime of the Arnoldi process for a particulate suspension problem with varying particle concentration at $N=40^3$.}\label{sphere}
\label{Mhalf}
\begin{tabular}{c|c|c|c}
\hline$d$ 
& Iteration Counts $k$ & relative error & runtime \\
\hline 
$3$  &12 &1.93793e-04 &362.4649 \\
$3.5$ &12 &1.81240e-04 &386.0504\\
$4$  &10 &1.88017e-04 &319.9872 \\
$4.5$  &10 &1.86118e-04 &307.9138 \\
$5$  &10 &1.83476e-04 &298.9904 \\

\hline
\end{tabular}
\end{center}
\end{table}

\section{Concluding remarks}\label{sum}
In this work, we derive error estimates of the Arnoldi process for approximating the matrix square root applied to a vector, extending the results reported in \cite{chen2022error} for Hermitian matrices to non-Hermitian cases.  In addition, we also provide a sharper bound for the Hermitian case, improving the state-of-the-art result. Moreover, we provide an error bound based on a perturbed matrix. Through a series of numerical experiments, we systematically assess the convergence behaviors of the Arnoldi iteration for computing $M^{1/2} \bm{b}$ across different matrix types. The findings indicate that the theoretical results presented in  \cref{error_square} appear to be sharp in terms of iteration count, yet remain a loose upper bound with respect to the largest singular value for certain classes of matrices. Furthermore, we consider the mobility matrix arising from particulate suspension simulations, where the matrix is approximated into a hierarchical form. 
The numerical results demonstrate the effectiveness of our approach for large-scale problems.
Future extensions of this work include deriving a sharper dependence on the largest singular value for approximating the matrix square root applied to a vector via Arnoldi iteration, as well as investigating the error behavior of GMRES-type methods in the approximation of $f(M)\bm{b}$ for general functions $f$.

\section*{Acknowledgments}
The authors would like to thank Zhan Ma and Zisheng Ye for their valuable discussions and assistance with the code implementation. The authors also gratefully acknowledge funding support from the National Science Foundation under Grant No. DMS-2208267 and the Army Research Office under Grant No. W911NF2310256.
\section*{Conflicts of Interest}
The authors declare no conflicts of interest.
\section*{Data Availability Statement}
The MATLAB codes used in this work are available from the corresponding author upon reasonable request.

\bibliographystyle{wileyj}
\bibliography{new_refs}

\end{document}